\documentclass[11pt]{amsart}

\usepackage{fullpage} 
\usepackage[foot]{amsaddr} 
\usepackage{graphicx} 
\usepackage{amsmath, amssymb, amsfonts, amsthm, mathtools} 
\usepackage{enumerate} 

\usepackage{color} 
\usepackage{url} 
\usepackage[plainpages = false, colorlinks, citecolor = blue, filecolor = black, linkcolor = red, urlcolor = cyan]{hyperref} 

\theoremstyle{definition} 
\newtheorem{definition}{Definition}[section] 
\theoremstyle{plain} 
\newtheorem{theorem}[definition]{Theorem}

\newtheorem{lemma}[definition]{Lemma}
\newtheorem{corollary}[definition]{Corollary} 
\theoremstyle{remark}
\newtheorem{remark}[definition]{Remark}
 
\allowdisplaybreaks 

\usepackage[backend = bibtex, giveninits = true, style = alphabetic, natbib = true, url = false, sorting = nyt, doi = true, abbreviate = true, backref = false]{biblatex}
\renewbibmacro{in:}{\ifentrytype{article}{}{\printtext{\bibstring{in}\intitlepunct}}}

\newcommand{\R}{\mathbb{R}} 
\newcommand{\C}{\mathbb{C}} 
\newcommand{\Z}{\mathbb{Z}} 
\newcommand{\N}{\mathbb{N}} 
\newcommand{\ie}{{\it{i.e.}, }} 
 
\newcommand{\vareps}{\varepsilon} 
\newcommand{\del}{\partial} 
\newcommand{\n}{\mathbf{n}} 
\newcommand{\tn}{\textnormal} 
\renewcommand{\vec}[1]{\hat{\boldsymbol{#1}}} 
\newcommand{\abs}[1]{\lvert{#1}\rvert} 
\newcommand{\mat}[6]{\begin{pmatrix} #1 & #2 & #3 \\ #4 & #5 & #6 \end{pmatrix}} 
\newcommand{\cc}{\overline} 
\newcommand{\trace}[1]{\operatorname{tr}\left(#1\right)} 
\newcommand{\diff}{\mathrm{diff}}

\begin{document} 
\title{Steklov Eigenvalues of Nearly Hyperspherical Domains} 
\author{Chee Han Tan}
\address{Department of Mathematics, Wake Forest University, Winston-Salem, NC} 
\email{tanch@wfu.edu} 

\author{Robert Viator} 
\address{Department of Mathematics, Denison University, Granville, OH} 
\email{viatorr@denison.edu} 

\subjclass[2010]{35C20, 35P05, 41A58} 
\keywords{Steklov eigenvalues, perturbation theory, hyperspherical harmonics, isoperimetric inequality} 

\date{\today} 

\begin{abstract} 
We consider Steklov eigenvalues of nearly hyperspherical domains in $\R^{d + 1}$ with $d\ge 3$. In previous work, treating such domains as perturbations of the ball, we proved that the Steklov eigenvalues are analytic functions of the domain perturbation parameter. Here, we compute the first-order term of the asymptotic expansion and show that the first-order perturbations are eigenvalues of a Hermitian matrix, whose entries can be written explicitly in terms of the Pochhammer's and Wigner $3j$-symbols. We analyse the asymptotic expansion and show the following isoperimetric results among domains with fixed volume: (1) for an infinite subset of Steklov eigenvalues, the ball is not optimal, and (2) for a different infinite subset of Steklov eigenvalues, the ball is a stationary point. 
\end{abstract} 

\maketitle


\section{Introduction} 
Let $\Omega\subset\R^{d + 1}$ be a bounded domain with $d\ge 1$. The Steklov eigenvalue problem for $(\lambda, u)$ on $\Omega$ is given by 
\begin{subequations} 
\begin{alignat}{2} 
\label{eq:Steklov1} \Delta u & = 0 && \ \ \tn{ in } \Omega, \\ 
\label{eq:Steklov2} \del_\n u & = \lambda u && \ \ \tn{ on } \del\Omega, 
\end{alignat} 
\end{subequations} 
where $\Delta$ is the Laplacian acting on $H^1(\Omega)$, $\del_\n u = \nabla u\cdot \n$ is the unit outward normal derivative on the boundary $\del\Omega$, and $\lambda$ is the spectral parameter. It is well-known that the Steklov spectrum is discrete as long as the trace operator $T\colon H^1(\Omega)\to L^2(\del\Omega)$ is compact \cite{girouard2017}. Moreover, the eigenvalues are real and we enumerate them, counting multiplicity, in increasing order 
\begin{equation*} 
0 < \lambda_0(\Omega) < \lambda_1(\Omega)\le \lambda_2(\Omega)\le \dots \nearrow\infty. 
\end{equation*} 
The Steklov eigenvalue problem has received considerable attention in the literature; see the survey papers \cite{girouard2017, colbois2023} and references therein. 

The Steklov eigenvalue problem was first introduced by Vladimir Steklov in \cite{stekloff1902} to describe the stationary heat distribution in a body $\Omega$ whose heat flux through the boundary is proportional to the temperature. For planar domains, the Steklov eigenvalues are the squares of the natural frequencies of a vibrating free membrane with all its mass concentrated along the boundary \cite[p. 95]{Bandle}. Steklov eigenvalues also have applications in optimal material design for both electromagnetism and torsional rigidity \cite{lipton1998a, lipton1998b}. Recently, Cakoni et al. \cite{cakoni2016} used Steklov eigenvalues in nondestructive testing, where they established a crucial relationship between small changes in the (possibly complex valued) refractive index of a scattering object and the corresponding change in the eigenvalue of a modified Steklov problem. For this problem, numerical results in \cite{cakoni2016} revealed that a localised defect of the refractive index in a disc perturbs only a small number of modified Steklov eigenvalues. 

Isoperimetric inequalities for Steklov eigenvalues have been explored since the mid-twentieth century. The first major result was obtained by Weinstock in his 1954 seminal paper \cite{weinstock1954}, where he showed that the disc uniquely maximises the first nontrivial perimeter-normalised Steklov eigenvalue $\lambda_1(\Omega)\abs{\del\Omega}$ among all bounded simply connected planar domains with smooth boundary. For higher eigenvalues, Girouard and Polterovich showed that the $n$th perimeter-normalised Steklov eigenvalue is maximised in the limit by a sequence of simply connected planar domains degenerating to the disjoint union of $n$ identical discs. At the same time, it is known that Weinstock's result fails for non simply-connected planar domains \cite[Example 4.2.5]{girouard2017}. In dimension 3 or higher, Fraser and Schoen \cite{fraser2019} showed that Weinstock's result fails for general contractible domains, but Bucur et al. \cite[Theorem 3.1]{bucur2021} showed that Weinstock's result holds for all bounded convex domains with Lipschitz boundary. 

While it is natural to consider the maximisation of Steklov eigenvalues with prescribed perimeter (because the spectral parameter $\lambda$ appears on the boundary $\del\Omega$), in this paper we focus on the maximisation of Steklov eigenvalues with prescribed volume. For $\Omega\subset\R^{d + 1}$, let $\Lambda(\Omega)\coloneqq \lambda(\Omega)\cdot\abs{\Omega}^{\frac{1}{d + 1}}$ denote the volume-normalised Steklov eigenvalue. Brock proved that the ball uniquely maximises $\Lambda_1(\Omega)$ for bounded Lipschitz domains $\Omega\subset\R^{d + 1}$ in all dimensions. For higher eigenvalues, Bogosel, Bucur, and Giacomini \cite{bogosel2017} obtained the existence and regularity results for the shape optimiser for $\Lambda_n(\Omega)$, $n\ge 2$, on bounded Lipschitz domains. In dimension 2, numerical results from \cite{bogosel2016, bogosel2017, akhmetgaliyev2017} suggested that the optimal domain is unique (up to dilations and rigid  transformations), has $n$-fold symmetry, and has at least one axis of symmetry. In particular, the ball is not a maximiser for an infinite subset of Steklov eigenvalues for planar domains; this was confirmed for reflection-symmetric domains in \cite{viator2018}. 

Motivated by the asymptotic work of Lord Rayleigh \cite{Rayleigh} and Wolf and Keller \cite{wolf1994} on the minimisation of Laplace-Dirichlet eigenvalues on planar domains, Viator and Osting adopted  their perturbative approach to study Steklov eigenvalues on reflection-symmetric nearly circular planar domains \cite{viator2018} and nearly spherical domains \cite{viator2022}. In dimension 3, Viator and Osting \cite[Theorem 1.1]{viator2022} proved that for $n = 1, 2, \dots$, $\Lambda_{(n + 1)^2 - 1}$ is not maximised by the ball but $\Lambda_{n^2}$ is stationary for a ball, suggesting that the ball is a natural candidate for maximiser of $\Lambda_{n^2}$. However, recent numerical results from Antunes \cite{antunes2021} suggest that the ball maximises $\Lambda_4$ but not $\Lambda_9$ and $\Lambda_{16}$. The same numerical results also suggested that the optimal domain for $\Lambda_n$ seems to have $n$ "buds" and some of the optimal domains seem to have symmetries that can be related with Platonic solids. 

Tuning of mixed Steklov-Neumann boundary conditions have also been recently studied by Ammari, Imeri, and Nigam \cite{ammari2020}, where an algorithm was designed to generate the proper mixed boundary conditions necessary to obtain desired resonance effects. Besides shape optimisation and isoperimetric results, there have been numerous recent results connecting Steklov eigenvalues to free-boundary minimal surfaces, inverse problems, and more; see \cite{colbois2023} for an extensive, though not exhaustive, review of recent work in Steklov eigenvalues.

\subsection{Nearly hyperspherical domains} 
Given $d\ge 3$, we consider the Steklov eigenvalue problem on a \emph{nearly hyperspherical domain} $\Omega_\vareps\subset\R^{d + 1}$ in hyperspherical coordinates $(r, \hat\theta)$, where $\Omega_\vareps$ has the form 
\begin{equation} \label{eq:Domain} 
\Omega_\vareps = \left\{(r, \hat\theta): 0\le r\le 1+ \vareps\rho(\hat\theta), \, \hat\theta \in S^d \right\}, \ \ \rho(\hat\theta) = \sum_{p = 0}^\infty\sum_{q = 1}^{N(d, p)} A_{p, q} Y_{p, q}(\hat\theta). 
\end{equation} 
Here, $\vareps\ge 0$ is a small perturbation parameter, $\rho\in C^1(S^d)$ is a perturbation function which we expand in the basis of real hyperspherical harmonics (see Section \ref{sec:Harmonics}), and $S^d\subset\R^{d + 1}$ is the $d$-dimensional unit sphere. For $\vareps = 0$, $\Omega_0$ is the $(d + 1)$-dimensional unit ball $B\subset\R^{d + 1}$ and the eigenvalues are nonnegative integers $\lambda_{\ell, m} = \ell$, with multiplicity $N(d, \ell)$ given by 
\begin{equation} \label{eq:Steklov_multiplicity} 
N(d, 0) = 1 \ \ \tn{ and } \ \ N(d, \ell) = \binom{d + \ell}{d} - \binom{d + \ell - 2}{d}, \ \ \ell\ge 1. 
\end{equation} 
The corresponding eigenfunctions in hyperspherical coordinates $(r, \hat\theta)$ are given by 
\begin{equation*} 
u_{\ell, m}(r, \hat\theta) = r^\ell Y_\ell^m(\hat\theta), \ \ \ell\in\N = \{0, 1, 2, \dots\}, \ \ 1\le m\le N(d, \ell), 
\end{equation*} 
where $Y_\ell^m(\hat\theta)$ is a \emph{complex hyperspherical harmonics} of degree $\ell$ on $S^d$.  

Viator and Osting proved that the Steklov eigenvalues $\lambda^\vareps$ of a nearly circular ($d = 1$) and nearly spherical ($d = 2$) domains are analytic with respect to $\vareps$ \cite{viator2020}. This analyticity result was recently extended to nearly-hyperspherical domains \cite{tan2023}. The proof relies on the fact that the Steklov eigenvalues can be interpreted as the eigenvalues of the Dirichlet-to-Neumann map $G_{\rho, \vareps}\colon H^{1/2}(\del\Omega_\vareps)\to H^{-1/2}(\del\Omega_\vareps)$.

\subsection{Main results} 
In previous work \cite{viator2018, viator2022}, Viator and Osting used perturbation methods to study the asymptotic expansion of Steklov eigenvalues $\lambda^\vareps$ for reflection symmetric, nearly circular domains and nearly spherical domains. Moreover, these asymptotic results were used to establish local versions of the isoperimetric inequalities for certain Steklov eigenvalues. In this paper we extend their results to nearly hyperspherical domains.   

Given $d\ge 3$, we recall the volume-normalised Steklov eigenvalue $\Lambda(\Omega)\coloneqq \lambda(\Omega)\cdot\abs{\Omega}^{\frac{1}{d + 1}}$. For $k\in\Z^+$, define the index $N_{d, k} = \sum_{\ell = 1}^{k - 1} N(d, \ell)$. 

\begin{theorem} \label{thm:Iso_stat} 
Let $d\ge 3$ and $k\in\Z^+$. Then $\Lambda_{1 + N_{d, k}}$ is stationary in the sense that, for every perturbation function $\rho\in C^2(S^d)$, the map $\vareps\mapsto \Lambda_{1 + N_{d, k}}(\Omega_\vareps)$ is nonincreasing in $\abs{\vareps}$ for $\abs{\vareps}$ sufficiently small. 
\end{theorem} 

Theorem \ref{thm:Iso_stat} suggests that the ball is a natural candidate for maximiser of $\Lambda_{1 + N_{d, k}}$ in dimensions $d + 1\ge 4$. However, the recent numerical result from Antunes for $d + 1 = 4$ suggested that $\Lambda_{1 + N_{3, 1}} = \Lambda_5$ is maximised by the ball. On the other hand, we show that the $(d + 1)$-dimensional ball doesn't maximise another infinite subset of Steklov eigenvalues. 

\begin{theorem} \label{thm:Iso_notball} 
Let $d\ge 3$ and $k\in\Z^+$. Then $\Lambda_{N_{d, k + 1}}$ is not maximised by the $(d + 1)$-dimensional ball. 
\end{theorem}

\subsection{Outline} 
This paper is organised as follows. We begin by reviewing hyperspherical coordinates, hyperspherical harmonics, and compute the first-order asymptotic expansions for geometrical quantities related to $\Omega_\vareps$ in Section \ref{sec:prelim}. In Section \ref{sec:asymptotic}, we derive the first-order asymptotic expansion for Steklov eigenvalues of $\Omega_\vareps$; see Theorem \ref{thm:Oeps_matrix}. Section \ref{sec:trace} and Section \ref{sec:Wigner} are devoted to proving Theorem \ref{thm:Iso_stat} and Theorem \ref{thm:Iso_notball}, respectively. In Section \ref{sec:Wigner}, we also include the asymptotic result for the special case where the domain perturbation function is given by $\rho = Y_{p, q}(\hat\theta)$; see Theorem \ref{thm:Ypq}.


\section{Preliminaries} \label{sec:prelim} 
In this section, we first review vector calculus in hyperspherical coordinates. We then define and record several important properties of the hyperspherical harmonics. In particular, we derive the addition theorem for the derivatives of hyperspherical harmonics, which will be crucial in proving Theorem \ref{thm:Iso_notball}. Finally, we compute the first-order asymptotic expansions for the volume of $\Omega_\vareps$ and the unit outward normal vector $\n_{\rho, \vareps}$ to $\del\Omega_\vareps$. 

\subsection{Hyperspherical coordinates in $\R^{d + 1}$} \label{sec:Coordinates} 
Let $(x_1, x_2, \dots, x_{d + 1})$ denote the $(d + 1)$-dimensional Cartesian coordinates. The $(d + 1)$-dimensional hyperspherical coordinates $(r, \theta_1, \theta_2, \dots, \theta_d)$ are defined by the following equations: 
\begin{align*} 
x_{d + 1} & = r\cos\theta_d, \\ 
x_d & = r\sin\theta_d\cos\theta_{d - 1}, \\ 
x_{d - 1} & = r\sin\theta_d\sin\theta_{d - 1}\cos\theta_{d - 2}, \\ 
\vdots & \qquad \ \ \vdots \qquad \ \ \vdots \qquad \ \ \vdots \\ 
x_3 & = r\sin\theta_d\sin\theta_{d - 1}\dots\cos\theta_2, \\ 
x_2 & = r\sin\theta_d\sin\theta_{d - 1}\dots\sin\theta_2\sin\theta_1, \\ 
x_1 & = r\sin\theta_d\sin\theta_{d - 1}\dots\sin\theta_2\cos\theta_1,  \\ 
\end{align*} 
where the azimuth $0\le \theta_1 = \phi< 2\pi$ and inclinations $0\le \theta_2, \theta_3, \dots, \theta_d\le\pi$ define a $(d + 1)$-dimensional sphere with radius $r\ge 0$. The hyperspherical coordinates are an orthogonal curvilinear coordinate system in $\R^{d + 1}$. Define $\theta_0\coloneqq r$. The associated metric tensor $g$ is diagonal with components 
\begin{equation} \label{eq:SCmetric} 
g_{ij} = \sum_{k = 1}^{d + 1} \frac{\del x_k}{d\theta_i}\frac{\del x_k}{\del\theta_j} = h_i^2\delta_{i,j}, \ \ 0\le i, j\le d, 
\end{equation} 
where the scale factors are given by $h_0 = 1$ and $h_i = r\prod_{k = i + 1}^d \sin\theta_k$ for $i = 1, 2, \dots, d$; the latter includes the empty product which gives $h_d = r$. Here, $\delta_{i, j}$ denotes the usual Kronecker delta.

Let $\vec{r}$, $\vec{\theta}_1$, $\vec{\theta}_2$, \dots, $\vec{\theta}_d$ be orthonormal hyperspherical basis vectors and define $\eta_j = h_j/r$ for $j = 1, 2, \dots, d$. The gradient operator in hyperspherical coordinates is given by 
\begin{equation*} 
\nabla = \frac{\del}{\del r}\vec{r} + \frac{1}{r}\nabla_{S^d}, 
\end{equation*} 
where $\nabla_{S^d}$ is the gradient on $S^d$:  
\begin{equation} \label{eq:SCgrad} 
\nabla_{S^d} = \sum_{j = 1}^d \frac{1}{\eta_j}\frac{\del}{\del\theta_j}\vec{\theta}_j. 
\end{equation} 
The Laplacian in hyperspherical coordinates is given by 
\begin{equation*} 
\Delta = \frac{1}{r^d}\frac{\del}{\del r}\left(r^d\frac{\del}{\del r}\right) + \frac{1}{r^2}\Delta_{S^d}, 
\end{equation*} 
where $\Delta_{S^d}$ is the spherical Laplacian (Laplace-Beltrami operator) on $S^d$: 
\begin{equation} \label{eq:SCbeltrami} 
\Delta_{S^d} = \sum_{j = 1}^d \frac{1}{\eta_j^2\sin^{j - 1}(\theta_j)}\frac{\del}{\del\theta_j}\left(\sin^{j - 1}(\theta_j)\frac{\del}{\del\theta_j}\right). 
\end{equation} 
The volume element in hyperspherical coordinates is given by $dV = r^d dr\, d\sigma_d$, where $d\sigma_d$ is the surface element over $S^d$: 
\begin{equation*} \label{eq:SCsurface} 
d\sigma_d(\hat\theta) = \left(\prod_{j = 2}^d \sin^{j - 1}(\theta_j)\right) d\theta_1\, d\theta_2\dots d\theta_d.  
\end{equation*} 

\begin{remark} 
Throughout this paper, we denote with $\del_r$ and $\del_j$ the partial derivative with respect to $r$ and $\theta_j$, respectively, for $j = 1, 2, \dots, d$. 
\end{remark} 

\subsection{Hyperspherical harmonics on $S^d$} \label{sec:Harmonics} 
For $\ell\in\N = \{0, 1, 2, \dots\}$, let $\mathbf{H}_\ell^d$ denote the space of all hyperspherical harmonics of order $\ell$ on $S^d$. The dimension of $\mathbf{H}_\ell^d$ is the same as the multiplicity $N(d, \ell)$ of the Steklov eigenvalue $\lambda = \ell$ of the $(d + 1)$-dimensional unit ball $B$. Let $\{Y_\ell^m\}_{m = 1}^{N(d, \ell)}$ be an orthonormal basis of $\mathbf{H}_\ell^d$ with respect to the complex inner product on $L^2(S^d)$. The spaces $\mathbf{H}_\ell^d$ are pairwise orthonormal, \ie 
\begin{equation*} 
\int_{S^d} Y_\ell^m(\hat\theta)\cc{Y_k^n(\hat\theta)}\, d\sigma_d = \delta_{\ell, k}\delta_{m, n}. 
\end{equation*} 
and the family $\{Y_\ell^m\}_{\ell \in \N, 1 \leq m \leq N(d,\ell)}$ forms a complete orthonormal basis of $L^2(S^d)$.
It is well-known that each $Y_\ell^m$ is an eigenfunction of the spherical Laplacian $\Delta_{S^d}$ corresponding to the eigenvalue $-\ell(\ell + d - 1)$, \ie 
\begin{equation} \label{eq:SH_eigs} 
\Delta_{S^d}Y_\ell^m(\hat\theta) = -\ell(\ell + d - 1)Y_\ell^m(\hat\theta), \ \ 1\le m\le N(d, \ell). 
\end{equation} 
Multiplying \eqref{eq:SH_eigs} with $\cc{Y_k^n}$ and integrating by parts over $S^d$, we obtain the following integral identity: 
\begin{equation} \label{eq:SH_eigs_grad} 
\int_{S^d} \nabla_{S^d}Y_\ell^m(\hat\theta)\cdot \nabla_{S^d}\cc{Y_\ell^n(\hat\theta)}\, d\sigma_d = \ell(\ell + d - 1)\delta_{m, n}, \ \ 1\le m, n\le N(d, \ell). 
\end{equation}

Let $P_n^\alpha$ and $C_n^{(\alpha)}$ denote the \emph{associated Legendre polynomial} and the \emph{Gegenbauer (ultrapsherical) polynomial} of degree $n$, respectively, which can be defined through the Rodrigues formulas (see \cite[Table.~18.5.1]{NIST} and \cite[Eqs.~6.27 \& 6.29]{Williams}):
\begin{align*} 
P_n^\alpha(z) & = \frac{(-1)^\alpha}{2^n n!}(1 - z^2)^{\alpha/2}\frac{d^{n + \alpha}}{dz^{n + \alpha}}\left(z^2 - 1\right)^n, \\ 
C_n^{(\alpha)}(z) & = \frac{(2\alpha)_n}{(-2)^n\left(\alpha + \frac{1}{2}\right)_n n!}\left(1 - z^2\right)^{-\alpha + \frac{1}{2}}\frac{d^n}{dz^n}\left(1 - z^2\right)^{n + \alpha - \frac{1}{2}}, \ \ \alpha > - \frac{1}{2}, \, \alpha\neq 0, 
\end{align*} 
where $(z)_n = \Gamma(z + n)/\Gamma(z)$ is the Pochhammer's symbol; see \cite[Eq.~5.2.5]{NIST}. We define the complex hyperspherical harmonics $Y_\ell^m$ of degree $\ell$ on $S^d$ as 
\begin{equation*} 
Y_\ell^m(\hat\theta) = \widetilde Y_{m_2}^{m_1}(\phi, \theta_2)\prod_{j = 3}^d Y(\theta_j; m_{j - 1}, m_j), 
\end{equation*} 
where $m\coloneqq \left(m_1, m_2, \cdots, m_{d - 1}\right)$ is any $(d - 1)$-tuple satisfying the inequality 
\begin{equation} \label{eq:SH_tuple}
0\le \abs{m_1}\le m_2\le \dots\le m_{d - 1}\le m_d\coloneqq\ell, 
\end{equation} 
and $\widetilde Y_{m_2}^{m_1}$ is the three-dimensional complex spherical harmonics
\begin{equation*} 
\widetilde Y_{m_2}^{m_1}(\phi, \theta_2) = \sqrt{\frac{(2m_2 + 1)}{4\pi}\frac{(m_2 - m_1)!}{(m_2 + m_1)!}}\, e^{im_1\phi}P_{m_2}^{m_1}(\cos\theta_2). 
\end{equation*}
The functions $Y(\theta_j; m_{j - 1}, m_j)$ are real-valued and they are defined by (see \cite[Section II]{wen1985})
\begin{equation*} 
Y(\theta_j; m_{j - 1}, m_j) = \frac{1}{\mu_j}\left(\sin\theta_j\right)^{m_{j - 1}}C_{m_j - m_{j - 1}}^{\left(m_{j - 1} + \frac{j - 1}{2}\right)}(\cos\theta_j), \ \ j = 3, 4, \dots, d, 
\end{equation*} 
where $\mu_j$ is the normalisation constant of $Y(\theta_j; m_{j - 1}, m_j)$ (with respect to the measure $\sin^{j - 1}(\theta_j)\, d\theta_j$) satisfying 
\begin{equation*} 
\mu_j^2 = \frac{4\pi\Gamma(m_j + m_{j - 1} + j - 1)}{2^{2m_{j - 1} + j}(m_j - m_{j - 1})!\left(m_j + \frac{j - 1}{2}\right)\Gamma^2\left(m_{j - 1} + \frac{j - 1}{2}\right)}. 
\end{equation*} 
Here, $\Gamma(z)$ is the standard Gamma function and we define $\Gamma^2(z) = \left[\Gamma(z)\right]^2$. The real hyperspherical harmonics $Y_{\ell, m}$ of degree $\ell$ on $S^d$ can be defined in the same way as the three-dimensional real spherical harmonics $\widetilde Y_{m_2, m_1}$, \ie  
\begin{equation*} 
Y_{\ell, m}(\hat\theta) = \widetilde Y_{m_2, m_1}(\phi, \theta_2)\prod_{j = 3}^d Y(\theta_j; m_{j - 1}, m_j), 
\end{equation*} 
where 
\begin{equation*} \label{eq:SH_real} 
\widetilde Y_{m_2, m_1}(\phi, \theta_2) = \begin{dcases} 
\, \frac{i}{\sqrt{2}}\left[\widetilde Y_{m_2}^{m_1}(\phi, \theta_2) - (-1)^{m_1}\widetilde Y_{m_2}^{-m_1}(\phi, \theta_2)\right] & \ \ \tn{ if } m_1 < 0, \\ 
\, \widetilde Y_{m_2}^{0}(\phi, \theta_2) & \ \ \tn{ if } m_1 = 0, \\ 
\, \frac{1}{\sqrt{2}}\left[\widetilde Y_{m_2}^{-m_1}(\phi, \theta_2) + (-1)^{m_1}\widetilde Y_{m_2}^{m_1}(\phi, \theta_2)\right] & \ \ \tn{ if } m_1 > 0. 
\end{dcases} 
\end{equation*} 
It is straightforward to verify that the set of real hyperspherical harmonics are pairwise orthonormal on $L^2(S^d)$. For notational simplicity, throughout this paper we will suppress the dependence of $\hat\theta$ on $\rho$ and hyperspherical harmonics whenever appropriate.

\begin{remark} 
Whenever we are counting all possible hyperspherical harmonics as $1\le m\le N(d, \ell)$ for a fixed degree $\ell\in\N$, this should be understood as counting over all tuples $m = \left(m_1, m_2, \dots, m_{d - 1}\right)$ satisfying the condition \eqref{eq:SH_tuple}. Take for instance $d = 3$ and $\ell = 2$. We then have $1\le m\le N(3, 2) = 9$ and the $9$ possible tuples $(m_1, m_2)$ satisfying $0\le \abs{m_1}\le m_2\le m_3 = \ell = 2$ are 
\begin{align*} 
(0, 0), \, (0, 1), \, (0, 2), \, (1, 1), \, (1, 2), \, (-1, 1), \, (-1, 2), \, (2, 2), \, (-2, 2). 
\end{align*} 
\end{remark} 

\begin{remark} 
For any $\ell\in\N$, we will assume that the index $m = 1$ corresponds to the trivial tuple $(0, 0, \dots, 0)$. With this convention, the constant hyperspherical harmonic is $Y_0^1 = Y_{0, 1} = \abs{S^d}^{-1/2}$. 
\end{remark}

Another important result about hyperspherical harmonics is the addition theorem (see \cite[Eq.~51]{wen1985}), which states that 
\begin{equation} \label{eq:SH_addition}
\sum_{m = 1}^{N(d, \ell)} Y_\ell^m(\hat\theta)\cc{Y_\ell^m(\hat\theta')} = K(d, \ell)C_\ell^{\left(\frac{d - 1}{2}\right)}(\vec{u}\cdot\vec{u}'), \ \ K(d, \ell)\coloneqq \frac{N(d, \ell)}{\abs{S^d}\, C_\ell^{\left(\frac{d - 1}{2}\right)}(1)}, 
\end{equation} 
for any unit vector $\vec{u}, \vec{u}'\in S^d$ with corresponding angular coordinates $\hat\theta, \hat\theta'$. Setting $\vec{u} = \vec{u}'$, we have $\vec{u}\cdot\vec{u} = 1$ and 
\begin{equation} \label{eq:SH_addition_same} 
\sum_{m = 1}^{N(d, \ell)} \abs{Y_\ell^m(\hat\theta)}^2 = \frac{N(d, \ell)}{\abs{S^d}}. 
\end{equation} 
We now establish the addition theorem for the partial derivatives of hyperspherical harmonics when $\vec{u} = \vec{u}'$. Our proof is inspired by \cite{winch1995}.

\begin{theorem} \label{thm:SH_addition_grad}
Let $d\ge 3$ and $\ell\in\N$. For all $j = 1, 2, \dots, d$, we have  
\begin{equation*} 
\sum_{m = 1}^{N(d, \ell)} \frac{1}{\eta_j^2}\abs{\del_jY_\ell^m(\hat\theta)}^2 = (d - 1)K(d, \ell)C_{\ell - 1}^{\left(\frac{d + 1}{2}\right)}(1). 
\end{equation*} 
\end{theorem} 
\begin{proof} 
For simplicity of notation, we write $C_\ell^{\left(\frac{d - 1}{2}\right)}(\vec{u}\cdot\vec{u}') = C(\vec{u}\cdot\vec{u}') = C(z)$. For any fixed $j = 1, 2, \dots, d$, differentiating \eqref{eq:SH_addition} with respect to $\theta_j$ first and then $\theta_j'$ yields 
\begin{equation} \label{eq:SH_addition_grad1}
\sum_{m = 1}^{N(d, \ell)} \del_jY_k^m(\hat\theta)\del_{j'}Y_k^m(\hat\theta') = K(d, \ell)\left[\frac{d^2C}{dz^2}\left[\vec{u}\cdot\del_{j'}\vec{u}'\right]\left[\del_j\vec{u}\cdot\vec{u}'\right] + \frac{dC}{dz}\left[\del_j\vec{u}\cdot \del_{j'}\vec{u}\right]\right]. 
\end{equation} 
In the case of $\vec{u} = \vec{u}'$, we know that $z = \vec{u}\cdot\vec{u} = 1$ and this implies $\vec{u}\cdot\del_j\vec{u} = 0$. Since $\vec{u}\in S^d$ can be written as $\vec{u} = \left(x_1, x_2, \dots, x_{d + 1}\right)/r$, computing $\abs{\del_j\vec{u}}^2$ gives 
\begin{equation*} 
\abs{\del_j\vec{u}}^2 = \frac{1}{r^2}\sum_{k = 1}^{d + 1} \left(\frac{\del x_k}{\del\theta_j}\right)^2 = \frac{h_j^2}{r^2} = \eta_j^2, 
\end{equation*} 
thanks to \eqref{eq:SCmetric}. Consequently, setting $\vec{u} = \vec{u}'$ in \eqref{eq:SH_addition_grad1} and rearranging yields  
\begin{align*} 
\sum_{m = 1}^{N(d, \ell)} \frac{1}{\eta_j^2}\abs{\del_jY_\ell^m(\hat\theta)}^2 & = K(d, \ell)\frac{dC}{dz}\bigg|_{z = 1} = K(d, \ell)\cdot (d - 1)C_{\ell - 1}^{\left(\frac{d + 1}{2}\right)}(1),  
\end{align*} 
where we use the derivative formula \cite[Eq.~18.9.19]{NIST}. The desired result now follows. 
\end{proof} 

\subsection{Asymptotic expansions for geometric quantities} 
Let $\abs{S^d}$ and $\abs{B} = \abs{S^d}/(d + 1)$ denote the surface area of $S^d$ and the volume of the $(d + 1)$-dimensional unit ball $B$, respectively. Using the orthogonality of hyperspherical harmonics, we see from \eqref{eq:Domain} that 
\begin{equation*} 
\int_{S^d} \rho(\hat\theta)\, d\sigma_d = \int_{S^d} A_{0, 1}Y_{0, 1}\, d\sigma_d = A_{0, 1}\abs{S^d}^{1/2}. 
\end{equation*} 
Thus, an asymptotic expansion for the volume of $\Omega_\vareps$ is given by
\begin{align*} 
\abs{\Omega_\vareps} = \int_{S^d} \int_0^{1 + \vareps\rho(\hat\theta)} r^d\, dr\, d\sigma_d & = \frac{1}{d + 1}\int_{S^d} \left(1 + \vareps\rho(\hat\theta)\right)^{d + 1}\, d\sigma_d \\ 
& = \frac{\abs{S^d}}{d + 1} + \vareps\int_{S^d} \rho(\hat\theta)\, d\sigma_d + O(\vareps^2) \\ 
& = \abs{B} + \vareps A_{0, 1}\abs{S^d}^{1/2} + O(\vareps^2). 
\end{align*} 
In particular, we have that 
\begin{equation} \label{eq:Asymp_volume} 
\begin{aligned}  
\abs{\Omega_\vareps}^{\frac{1}{d + 1}} & = \abs{B}^{\frac{1}{d + 1}} + \vareps\left(\frac{1}{d + 1}\abs{B}^{\frac{1}{d + 1} - 1}A_{0, 1}\abs{S^d}^{1/2}\right) + O(\vareps^2) \\ 
& = \abs{B}^{\frac{1}{d + 1}} + \vareps\left(\frac{A_{0, 1}\abs{B}^{\frac{1}{d + 1}}}{\abs{S^d}^{1/2}}\right) + O(\vareps^2). 
\end{aligned} 
\end{equation}

Next we find an asymptotic expansion for the unit outward normal vector $\n_{\rho, \vareps}$ to $\del\Omega_\vareps$. By identifying $\del\Omega_\vareps$ as the zero level set of an implicit function, it can be shown that (see \cite[Section 5]{tan2023}) 
\begin{equation*} 
\n_{\rho, \vareps} = \Big((1 + \vareps\rho)^2 + \vareps^2\abs{\nabla_{S^d}\rho}^2\Big)^{-1/2}\Big[(1 + \vareps\rho)\vec{r} - \vareps\nabla_{S^d}\rho\Big]. 
\end{equation*} 
It follows that 
\begin{equation} \label{eq:Asymp_normal} 
\n_{\rho, \vareps} = \Big(1 - \vareps\rho + O(\vareps^2)\Big)\Big[\vec{r} + \vareps\rho\, \vec{r} - \vareps\nabla_{S^d}\rho\Big] = \vec{r} - \vareps\nabla_{S^d}\rho + O(\vareps^2). 
\end{equation}


\section{An Asymptotic Expansion for Steklov Eigenvalues of Nearly Hyperspherical Domains} \label{sec:asymptotic} 
In this section, we derive an asymptotic expansion for the Steklov eigenvalues $\lambda(\vareps)\coloneqq \lambda^\vareps$ on a nearly hyperspherical domain $\Omega_\vareps$ of the form in \eqref{eq:Domain}. Recall that the unperturbed eigenvalues of $\Omega_0 = B$ are the nonnegative integers $\ell\in\N$, with corresponding eigenfunctions $r^\ell Y_\ell^m(\hat\theta)$, $1\le m\le N(d, \ell)$. Following \cite{viator2022}, for fixed positive integer $k\in\Z^+$, we make the following perturbation ansatz in $\vareps$ for a Steklov eigenpair $(\lambda_k^\vareps, u_k^\vareps)$ of $\Omega_\vareps$ (not counting multiplicity): 
\begin{subequations} \label{eq:Pert_ansatz} 
\begin{align} 
\label{eq:Pert_ansatz1} \lambda_k^\vareps & = k + \vareps \lambda_k^{(1)} + O(\vareps^2), \\ 
\label{eq:Pert_ansatz2} u_k^\vareps(r, \hat\theta) & = \sum_{\ell = 0}^\infty\sum_{m = 1}^{N(d, \ell)} \Big(\delta_{\ell, k}\alpha_m + \vareps\beta_{\ell, m} + O(\vareps^2)\Big)r^\ell Y_\ell^m(\hat\theta). 
\end{align} 
\end{subequations} 
Note that we cannot apriori determine the coefficients $\alpha_m$ that will select the $O(1)$ eigenfunction from the $N(d, k)$-dimensional eigenspace.

The ansatz \eqref{eq:Pert_ansatz2} satisfies \eqref{eq:Steklov1} exactly and we will determine the eigenvalue perturbation $\lambda_k^{(1)}$ and the coefficients $\alpha_m$ and $\beta_{\ell, m}$ so that the boundary condition \eqref{eq:Steklov2} is satisfied. Using the gradient \eqref{eq:SCgrad} in hyperspherical coordinates, we have that 
\begin{equation} \label{eq:Pert_grad1} 
\nabla u_k^\vareps(r, \hat\theta) = \sum_{\ell = 0}^\infty\sum_{m = 1}^{N(d, \ell)} \Big(\delta_{\ell, k}\alpha_m + \vareps\beta_{\ell, m} + O(\vareps^2)\Big) r^{\ell - 1}\, \vec{v}_{\ell, m}, 
\end{equation} 
where 
\begin{equation} \label{eq:Pert_grad2} 
\vec{v}_{\ell, m} = \ell Y_\ell^m\vec{r} + \nabla_{S^d}Y_\ell^m. 
\end{equation} 
The boundary condition \eqref{eq:Steklov2} reads 
\begin{equation} \label{eq:Pert_BC} 
\nabla u_k^\vareps\cdot \n_{\rho, \vareps} = \lambda_k^\vareps u_k^\vareps \ \ \tn{ on } r = 1 + \vareps\rho(\hat\theta). 
\end{equation} 
Substituting \eqref{eq:Pert_grad1} and the asymptotic expansion \eqref{eq:Asymp_normal} for $\n_{\rho, \vareps}$ into the left-hand side (LHS) of \eqref{eq:Pert_BC} and collecting terms in powers of $\vareps$, we obtain 
\begin{equation*} 
\nabla u_k^\vareps\cdot\n_{\rho, \vareps} = \left(\sum_{m = 1}^{N(d, k)} k\alpha_m Y_k^m\right) + \vareps L_1 + O(\vareps^2),
\end{equation*} 
where 
\begin{align*} 
L_1 & = \sum_{\ell = 0}^\infty\sum_{m = 1}^{N(d, \ell)} \Big(\delta_{\ell, k}\alpha_m\left((\ell - 1)\rho\, \vec{r} - \nabla_{S^d}\rho\right) + \beta_{\ell, m}\vec{r}\Big)\cdot\vec{v}_{\ell, m} \\ 
& \stackrel{\eqref{eq:Pert_grad2}}{=} \sum_{m = 1}^{N(d, k)} \alpha_m\Big(k(k - 1)\rho Y_k^m - \nabla_{S^d}\rho\cdot\nabla_{S^d}Y_k^m\Big) + \sum_{\ell = 0}^\infty\sum_{m = 1}^{N(d, \ell)} \ell \beta_{\ell, m}Y_\ell^m. 
\end{align*} 
Substituting the perturbation ansatz \eqref{eq:Pert_ansatz} into the right-hand side (RHS) of \eqref{eq:Pert_BC} and collecting terms in powers of $\vareps$, we obtain 
\begin{equation*} 
\lambda_k^\vareps u_k^\vareps = \left(\sum_{m = 1}^{N(d, k)} k\alpha_m Y_k^m\right) + \vareps R_1 + O(\vareps^2),
\end{equation*} 
where 
\begin{align*} 
R_1 & = \sum_{\ell = 0}^\infty\sum_{m = 1}^{N(d, \ell)} \Big(k\left(\beta_{\ell, m} + \delta_{\ell, k}\alpha_m\ell\rho\right) + \lambda_k^{(1)}\delta_{\ell, k}\alpha_m\Big)Y_\ell^m \\ 
& = \sum_{\ell = 0}^\infty\sum_{m = 1}^{N(d, \ell)} k\beta_{\ell, m}Y_\ell^m + \sum_{m = 1}^{N(d, k)} \Big(k^2\rho + \lambda_k^{(1)}\Big)\alpha_m Y_k^m. 
\end{align*} 
The $O(1)$ terms in the LHS and RHS of \eqref{eq:Pert_BC} coincide, as expected. Rearranging the $O(\vareps)$ equation $L_1 = R_1$, we obtain 
\begin{equation} \label{eq:Pert_Oeps}  
\sum_{m = 1}^{N(d, k)} \lambda_k^{(1)}\alpha_mY_k^m = -\sum_{m = 1}^{N(d, k)} \alpha_m\Big(k\rho Y_k^m + \nabla_{S^d}\rho\cdot \nabla_{S^d}Y_k^m\Big) + \sum_{\ell = 0}^\infty\sum_{m = 1}^{N(d, \ell)} (\ell - k)\beta_{\ell, m}Y_k^m.  
\end{equation} 

If we now multiply \eqref{eq:Pert_Oeps} by $\cc{Y_k^n}$ for $1\le n\le N(d, k)$, integrate over $S^d$ with respect to $d\sigma_d$, and use the pairwise orthonormality of the hyperspherical harmonics, we see that the resulting sum on the left is nonzero only for $m = n$ and the third sum vanish for all $m$. This yields 
\begin{equation*} 
\lambda_k^{(1)}\alpha_n = \sum_{m = 1}^{N(d, k)} M_{m, n}^{(d, k)}\alpha_m, \ \ 1\le n\le N(d, k), 
\end{equation*} 
or more succinctly, $M^{(d, k)}\vec{\alpha} = \lambda_k^{(1)}\vec{\alpha}$, where the complex matrix $M^{(d, k)}\in\C^{N(d, k)\times N(d, k)}$ has entries given by
\begin{gather} \label{eq:Oeps_matrix}
\begin{aligned}  
M_{m, n}^{(d, k)} & = -\int_{S^d} k\rho Y_k^m\cc{Y_k^n}\, d\sigma_d - \int_{S^d} \left(\nabla_{S^d}\rho\cdot \nabla_{S^d}Y_k^m\right)\cc{Y_k^n}\, d\sigma_d \\ 
\end{aligned} 
\end{gather} 
This shows that the first-order perturbation of these $N(d, k)$ eigenvalues are characterised by the eigenvalues of $M^{(d, k)}$. Since $M^{(d, k)}$ is Hermitian, there are $N(d, k)$ real eigenvalues $\lambda_{k, j}^{(1)}$, $1\le j\le N(d, k)$, which we enumerate in increasing order. Moreover, the components of the corresponding eigenvectors $\vec{\alpha}_j$ are the coefficients of the $O(1)$ eigenfunctions in $u_{k, j}^\vareps$, \ie 
\begin{align*} 
u_{k, j}^{(0)}(r, \hat\theta) = \sum_{m = 1}^{N(d, k)} \left(\vec{\alpha}_j\right)_m r^kY_k^m(\hat\theta), \ \ 1\le j\le N(d, k). 
\end{align*} 
We summarise these results and the analyticity result from \cite{tan2023} in the following theorem.

\begin{theorem} \label{thm:Oeps_matrix} 
Given $d\ge 3$ and $k\in\Z^+$, define $N_k = \sum_{\ell = 1}^{k - 1} N(d, \ell)$, where $N(d, \ell)$ is defined by \eqref{eq:Steklov_multiplicity}. For $N_k + 1\le n\le N_{k + 1}$, the Steklov eigenvalues $\lambda_n(\vareps)$ of a nearly hyperspherical domain $\Omega_\vareps$ of the form \eqref{eq:Domain} consist of at most $N(d, k)$ branches of analytic functions which have at most algebraic singularities near $\vareps = 0$. At first-order in $\vareps$, the perturbation is given by the real eigenvalues of the Hermitian matrix $M^{(d, k)}$ of size $N(d, k)$, whose entries are given by \eqref{eq:Oeps_matrix}. 
\end{theorem}


\section{Analysis of $M^{(d, k)}$} \label{sec:trace} 
This section is dedicated to the proof of Theorem \ref{thm:Iso_notball}. Following \cite[Theorem 1.1]{viator2022}, the crux of the proof lies in showing that the trace of $M^{(d, k)}$ is proportional to $\int_{S^d} \rho\, d\sigma_d$, the mean of the domain perturbation function $\rho$. This is achieved by rewriting $M^{(d,k)}$ as the sum of a scalar multiple of the identity and a trace-zero Hermitian matrix. For notational simplicity, we write the surface element over $S^d$ as $d\sigma_d = \sin^{j - 1}(\theta_j)\, d\theta_j\, d\sigma_{d - 1, j}$, where 
\begin{equation*} 
d\sigma_{d - 1, j} = \prod_{i = 2, i\neq j}^d \sin^{i - 1}(\theta_i)\prod_{i = 1, i\neq j}^d d\theta_i. 
\end{equation*}

\begin{lemma} \label{thm:Oeps_matrix1} 
Given $d\ge 3$ and $k\in\Z^+$, let $M^{(d, k)}$ be the Hermitian matrix defined in Theorem \ref{thm:Oeps_matrix}. The entries of $M^{(d, k)}$ can be written as 
\begin{equation} \label{eq:Oeps_matrix1} 
M_{m, n}^{(d, k)} = \int_{S^d} \rho(\hat\theta)\left(-k(k + d)Y_k^m(\hat\theta)\cc{Y_k^n(\hat\theta)} + \nabla_{S^d}Y_k^m(\hat\theta)\cdot\nabla_{S^d}\cc{Y_k^n(\hat\theta)}\right) d\sigma_d. 
\end{equation} 
\end{lemma} 
\begin{proof} 
From Theorem \ref{thm:Oeps_matrix} and \eqref{eq:SCgrad}, we have 
\begin{equation} \label{eq:Oeps_matrix1a} 
M_{m, n}^{(d, k)} = -\int_{S^d} k\rho Y_k^m\cc{Y_k^n}\, d\sigma_d - \sum_{j = 1}^d \int_{S^d} \frac{\del_j\rho}{\eta_j^2}\del_jY_k^m\cdot \cc{Y_k^n}\, d\sigma_d.
\end{equation} 
For each integral from the sum above, we integrate by parts with respect to $\theta_j$. Note crucially that $\eta_j$ is independent of $\theta_j$. For the case $j = 1$, the determinant of the Jacobian in $d\sigma_d$ is independent of $\theta_1$ and the boundary term vanishes due to the $2\pi$-periodicity of $\rho, \del_1Y_k^m, Y_k^n$ with respect to $\theta_1$. This yields 
\begin{align*} 
-\int_{S^d} \frac{\del_1\rho}{\eta_1^2}\del_1Y_k^m\cdot \cc{Y_k^n}\, d\sigma_d & = \int_{S^d} \frac{\rho}{\eta_1^2}\del_1\left(\del_1Y_k^m\cdot \cc{Y_k^n}\right)d\sigma_d \\ 
& = \int_{S^d} \frac{\rho}{\eta_1^2}\Big[\del_1Y_k^m\cdot \del_1\cc{Y_k^n} + \del_1^2Y_k^m\cdot \cc{Y_k^n}\Big]\, d\sigma_d. 
\end{align*} 
For the case $j = 2, 3, \dots, d$, the boundary term vanishes because $\sin^{j - 1}(0) = \sin^{j - 1}(\pi) = 0$ for $j\ge 2$. This yields 
\begin{align*} 
& -\int_{S^{d - 1}}\int_{\theta_j} \frac{\del_j\rho}{\eta_j^2}\del_jY_k^m\cdot \cc{Y_k^n}\sin^{j - 1}(\theta_j)\, d\theta_j\, d\sigma_{d - 1, j} \\ 
& = \int_{S^{d - 1}}\int_{\theta_j} \frac{\rho}{\eta_j^2}\del_j\Big(\del_jY_k^m\cdot \cc{Y_k^n}\sin^{j - 1}(\theta_j)\Big) d\theta_j\, d\theta_{d - 1, j} \\ 
& = \int_{S^d} \frac{\rho}{\eta_j^2} \left[\del_jY_k^m\cdot \del_j\cc{Y_k^n} + \frac{\del_j\left(\sin^{j - 1}(\theta_j)\del_jY_k^m\right)}{\sin^{j - 1}(\theta_j)}\, \cc{Y_k^n}\right] d\sigma_d.  
\end{align*} 
Summing over all $j = 1, 2, \dots, d$ and recalling the gradient \eqref{eq:SCgrad} and the spherical Laplacian \eqref{eq:SCbeltrami} on $S^d$, we obtain 
\begin{gather} \label{eq:Oeps_matrix1b} 
\begin{aligned} 
& -\sum_{j = 1}^d \int_{S^d} \frac{\del_j\rho}{\eta_j^2}\del_jY_k^m\cdot \cc{Y_k^n}\, d\sigma_d \\ 
& = \sum_{j = 1}^d \int_{S^d} \rho\frac{\del_jY_k^m}{\eta_j}\cdot \frac{\del_j\cc{Y_k^n}}{\eta_j}\, d\sigma_d + \sum_{j = 1}^d \int_{S^d} \rho \left(\frac{\del_j\left(\sin^{j - 1}(\theta_j)\del_jY_k^m\right)}{\eta_j^2\sin^{j - 1}(\theta_j)}\right) \cc{Y_k^n}\, d\sigma_d \\ 
& = \int_{S^d} \rho\nabla_{S^d}Y_k^m\cdot \nabla_{S^d}\cc{Y_k^n}\, d\sigma_d + \int_{S^d} \rho\left(\Delta_{S^d}Y_k^m\right)\cc{Y_k^n}\, d\sigma_d \\ 
& \stackrel{\eqref{eq:SH_eigs}}{=} \int_{S^d} \rho\nabla_{S^d}Y_k^m\cdot \nabla_{S^d}\cc{Y_k^n}\, d\sigma_d - \int_{S^d} k(k + d - 1)\rho Y_k^m\cc{Y_k^n}\, d\sigma_d. 
\end{aligned} 
\end{gather} 
Substituting \eqref{eq:Oeps_matrix1b} into \eqref{eq:Oeps_matrix1a} and rearranging gives the desired expression \eqref{eq:Oeps_matrix1} for $M_{m, n}^{(d, k)}$. 
\end{proof}

We are now ready to prove that the trace of $M^{(d, k)}$ is proportional to the mean of $\rho$. 

\begin{lemma} \label{thm:Matrix_trace} 
Given $d\ge 3$ and $k\in\Z^+$, let $M^{(d, k)}$ and $\rho$ be the Hermitian matrix and the perturbation function defined in Theorem \ref{thm:Oeps_matrix} and \eqref{eq:Domain}, respectively. The trace of $M^{(d, k)}$ is given by 
\begin{equation*} 
\trace{M^{(d, k)}} = -\frac{k N(d, k)}{\abs{S^d}}\int_{S^d} \rho \, d\sigma_d = -\frac{kA_{0, 1}N(d, k)}{\abs{S^d}^{1/2}}. 
\end{equation*} 
\end{lemma} 
\begin{proof} 
From Lemma \eqref{thm:Oeps_matrix1}, we have 
\begin{equation*} 
\trace{M^{(d, k)}} = \sum_{m = 1}^{N(d, k)} M_{m ,m}^{(k)} = \sum_{m = 1}^{N(d, k)} \int_{S^d} \rho\Big(-k(k + d)Y_k^m\cc{Y_k^m} + \nabla_{S^d}Y_k^m\cdot\nabla_{S^d}\cc{Y_k^m}\Big)\, d\sigma_d. 
\end{equation*} 
The lemma is an direct consequence of the addition theorem for hyperspherical harmonics and its gradient; see \eqref{eq:SH_addition}, \eqref{eq:SH_addition_same}, and Theorem \ref{thm:SH_addition_grad}. Indeed, 
\begin{align*} 
\trace{M^{(d, k)}} & = \int_{S^d} \rho\left(-\frac{k(k + d)N(d, k)}{\abs{S^d}} + d(d - 1)K(d, k)C_{\ell - 1}^{\left(\frac{d + 1}{2}\right)}(1)\right) d\sigma_d \\ 
& = -\frac{N(d, k)}{\abs{S^d}}\left[-k(k + d) + d(d - 1)\frac{C_{k - 1}^{\left(\frac{d + 1}{2}\right)}(1)}{C_k^{\left(\frac{d - 1}{2}\right)}(1)}\right] \int_{S^d}\rho\, d\sigma_d. 
\end{align*}  
We need only show the expression in the bracket above is equal to $k$. From \cite[Table.~18.6.1]{NIST} and the definition of Pochhammer's symbol \cite[Eq.~5.2.5]{NIST}, we have that 
\begin{equation*} 
C_n^{(\alpha)}(1) = \frac{(2\alpha)_n}{n!} = \frac{\Gamma(2\alpha + n)}{n!\Gamma(2\alpha)}. 
\end{equation*} 
Consequently, 
\begin{align*} 
d(d - 1)\cdot \frac{C_{k - 1}^{\left(\frac{d + 1}{2}\right)}(1)}{C_k^{\left(\frac{d - 1}{2}\right)}(1)} & = \frac{\Gamma(d + k)}{\Gamma(d + 1)(k - 1)!}\cdot \frac{d(d - 1)\Gamma(d - 1)k!}{\Gamma(d + k - 1)} = k(d + k - 1),  
\end{align*} 
where we use the fact that $\Gamma(z + 1) = z\Gamma(z)$ for any positive integer $z$. The claim now follows. 
\end{proof}

\begin{corollary} \label{thm:Matrix_sum} 
Given $d\ge 3$ and $k\in\Z^+$, let $M^{(d, k)}$ and $\rho$ be the Hermitian matrix and the perturbation function defined in Theorem \ref{thm:Oeps_matrix} and \eqref{eq:Domain}, respectively. We have 
\begin{equation} 
M^{(d, k)} = -\frac{kA_{0, 1}}{\abs{S^d}^{1/2}}I_{N(d, k)} + E^{(d, k)}, 
\end{equation} 
where $I_{N(d, k)}$ is the identity matrix of size $N(d, k)$ and $E^{(d, k)}$ is a Hermitian, zero-trace matrix, whose entries are given by 
\begin{equation*} 
E_{m, n}^{(d, k)} = \sum_{p = 1}^\infty\sum_{q = 1}^{N(d, p)} \int_{S^d} A_{p, q}Y_{p, q}\left(-k(k + d)Y_k^m(\hat\theta)\cc{Y_k^n(\hat\theta)} + \nabla_{S^d}Y_k^m(\hat\theta)\cdot\nabla_{S^d}\cc{Y_k^n(\hat\theta)}\right) d\sigma_d.
\end{equation*} 
\end{corollary} 
\begin{proof} 
We begin by substituting the expression for $\rho$ (see \eqref{eq:Domain}) into \eqref{eq:Oeps_matrix1} to obtain 
\begin{equation*} 
M_{m, n}^{(k)} = \sum_{p = 0}^\infty\sum_{q = 1}^{N(d, p)} \int_{S^d} A_{p, q}Y_{p, q}\left(-k(k + d)Y_k^m(\hat\theta)\cc{Y_k^n(\hat\theta)} + \nabla_{S^d}Y_k^m(\hat\theta)\cdot\nabla_{S^d}\cc{Y_k^n(\hat\theta)}\right) d\sigma_d. 
\end{equation*} 
Separating the infinite sum into $p = 0$ and $p > 0$, we may write $M_{m, n}^{(k)} = D_{m ,n}^{(d, k)} + E_{m, n}^{(k)}$, where $E_{m, n}^{(k)}$ has the desired expression and 
\begin{equation*} 
D_{m, n}^{(d, k)} = \int_{S^d} A_{0, 1}Y_{0, 1}\left(-k(k + d)Y_k^m(\hat\theta)\cc{Y_k^n(\hat\theta)} + \nabla_{S^d}Y_k^m(\hat\theta)\cdot\nabla_{S^d}\cc{Y_k^n(\hat\theta)}\right) d\sigma_d. 
\end{equation*} 
Using the integral identity \eqref{eq:SH_eigs_grad} and the orthonormality of the hyperspherical harmonics, we deduce that the matrix $D^{(d, k)}$ is diagonal. Moreover, 
\begin{align*} 
D_{m, m}^{(d, k)} & = A_{0, 1}Y_{0, 1}\Big(-k(k + d) + k(k + d - 1)\Big) = -\frac{kA_{0, 1}}{\abs{S^d}^{1/2}}, \ \ 1\le m\le N(d, k), 
\end{align*} 
as desired. Thanks to Lemma \ref{thm:Matrix_trace}, we see that $E^{(d, k)}$ has zero trace since the trace is linear. 
\end{proof} 

We are now ready to prove Theorem \ref{thm:Iso_stat}. 

\begin{proof}[Proof of Theorem \ref{thm:Iso_stat}]  
We recall the volume-normalised Steklov eigenvalue 
\begin{equation} 
\Lambda_{k, j}(\Omega_\vareps)= \lambda_{k, j}^{\vareps}\abs{\Omega_\vareps}^{\frac{1}{d + 1}}. 
\end{equation} 
Substituting the ansatz \eqref{eq:Pert_ansatz1} for $\lambda_{k, j}^\vareps$ and the asymptotic expansion \eqref{eq:Asymp_volume} for $\abs{\Omega_\vareps}^{\frac{1}{d + 1}}$, we obtain 
\begin{align*} 
\Lambda_{k, j}(\Omega_\vareps) & = \left(k + \vareps\lambda_{k, j}^{(1)} + O(\vareps^2)\right)\left(\abs{B}^{\frac{1}{d + 1}} + \vareps\left(\frac{A_{0, 1}\abs{B}^{\frac{1}{d + 1}}}{\abs{S^d}^{1/2}}\right) + O(\vareps^2)\right) \\ 
& = k\abs{B}^{\frac{1}{d + 1}} + \vareps\left(\lambda_{k, j}^{(1)}\abs{B}^{\frac{1}{d + 1}} + \frac{kA_{0, 1}\abs{B}^{\frac{1}{d + 1}}}{\abs{S^d}^{1/2}}\right) + O(\vareps^2). 
\end{align*} 
From Corollary \ref{thm:Matrix_sum}, we have that 
\begin{equation*} 
\lambda_{k, j}^{(1)} = -\frac{kA_{0, 1}}{\abs{S^d}^{1/2}} + e_{k, j}, 
\end{equation*} 
where $e_{k, j}$ is the $j$th eigenvalue (in increasing order) of the matrix $E^{(d, k)}$ which is real. It follows that 
\begin{align*} 
\Lambda_{k, j}(\Omega_\vareps) = k\abs{B}^{\frac{1}{d + 1}} + \vareps\left(e_{k, j}\abs{B}^{\frac{1}{d + 1}}\right) + O(\vareps^2). 
\end{align*} 
Since $E^{(d, k)}$ is Hermitian with zero trace, either $e_{k, j} = 0$ for all $j = 1, 2, \dots, N(d, k)$ or $e_{k, 1} < 0$. Together, we see that $e_{k, 1}\le 0$ and this completes the proof since $\Lambda_{k, 1} = \Lambda_{1 + N_{d, k}}$. 
\end{proof}


\section{$M^{(d, k)}$ and the Wigner $3j$-symbols} \label{sec:Wigner} 
To prove Theorem \ref{thm:Iso_notball}, it suffices to find a perturbation function $\rho$ such that the corresponding matrix $M^{(d, k)}$ has at least one positive eigenvalue. The first step is to express $M^{(d, k)}$ in terms of the integral of the triple product of hyperspherical harmonics.

\begin{lemma} \label{thm:Matrix_triple}
Given $d\ge 3$ and $k\in\Z^+$, let $M^{(d, k)}$ and $\rho$ be the Hermitian matrix and the perturbation function defined in Theorem \ref{thm:Oeps_matrix} and \eqref{eq:Domain}, respectively. Then the entries of $M^{(d, k)}$ can be written as 
\begin{equation} \label{eq:Matrix_triple} 
M_{m ,n}^{(d, k)} = -\frac{1}{2}\sum_{p = 0}^\infty\sum_{q = 1}^{N(d, k)} A_{p, q}\Big(p(p + d - 1) + 2k\Big)W_{q, m, n}^{p, k}, 
\end{equation} 
where 
\begin{equation*} 
W_{q, m, n}^{p, k} = \int_{S^d} Y_{p, q}(\hat\theta)Y_k^m(\hat\theta)\cc{Y_k^n(\hat\theta)}\, d\sigma_d. 
\end{equation*} 
\end{lemma} 
\begin{proof} 
From Theorem \ref{thm:Oeps_matrix} and \eqref{eq:SCgrad}, we have 
\begin{equation} \label{eq:Matrix_triple1} 
M_{m, n}^{(d, k)} \stackrel{\eqref{eq:Oeps_matrix1}}{=} -\int_{S^d} k\rho Y_k^m\cc{Y_k^n}\, d\sigma_d - \sum_{j = 1}^d\int_{S^d} \frac{\del_j\rho}{\eta_j^2}\del_jY_k^m\cdot \cc{Y_k^n}\, d\sigma_d. 
\end{equation}  
Integrating by parts with respect to $\theta_j$ and noting that $\eta_j$ is independent of $\theta_j$, we have that 
\begin{gather} 
\begin{aligned} \label{eq:Matrix_triple2}
-\sum_{j = 1}^d \int_{S^d} \frac{\del_j\rho}{\eta_j^2}\del_jY_k^m\cdot \cc{Y_k^n}\, d\sigma_d & = -\sum_{j = 1}^d \int_{S^{d - 1}}\int_{\theta_j} \frac{\del_jY_k^m}{\eta_j^2}\Big(\sin^{j- 1}(\theta_j)\del_j\rho\cdot \cc{Y_k^n}\Big)\, d\theta_j\, d\sigma_{d - 1, j} \\ 
& = \sum_{j = 1}^d \int_{S^{d - 1}}\int_{\theta_j} \frac{Y_k^m}{\eta_j^2}\del_j\Big(\sin^{j- 1}(\theta_j)\del_j\rho\cdot \cc{Y_k^n}\Big)\, d\theta_j\, d\sigma_{d - 1, j} \\ 
& = \sum_{j = 1}^d \int_{S^d} \frac{Y_k^m}{\eta_j^2}\left[\del_j\rho\cdot\del_j\cc{Y_k^n} + \frac{\del_j\left(\sin^{j - 1}(\theta_j)\del_j\rho\right)}{\sin^{j - 1}(\theta_j)}\cc{Y_k^n}\right] d\sigma_d \\ 
& \stackrel{\eqref{eq:SCbeltrami}}{=} \left(\sum_{j = 1}^d\int_{S^d} \frac{\del_j\rho}{\eta_j^2} Y_k^m\del_j\cc{Y_k^n}\, d\sigma_d\right) + \int_{S^d} \left(\Delta_{S^d}\rho\right) Y_k^m\cc{Y_k^n}\, d\sigma_d, 
\end{aligned} 
\end{gather} 
where the boundary term vanishes due to (1) $\del_1\rho$ and hyperspherical harmonics are $2\pi$-periodic with respect to $\theta_1$ for $j = 1$, and (2) $\sin^{j - 1}(0) = \sin^{j - 1}(\pi) = 0$ for $j = 2, 3, \dots, d$. On the other hand, we deduce from \eqref{eq:Oeps_matrix1b} from the proof of Theorem \ref{thm:Oeps_matrix} that 
\begin{equation} \label{eq:Matrix_triple3} 
-\sum_{j = 1}^d \int_{S^d} \frac{\del_j\rho}{\eta_j^2}\del_jY_k^m\cdot \cc{Y_k^n}\, d\sigma_d = -\sum_{j = 1}^d\int_{S^d} \frac{\del_j\rho}{\eta_j^2} Y_k^m\del_j\cc{Y_k^n}\, d\sigma_d. 
\end{equation} 
Taking the average of \eqref{eq:Matrix_triple2} and \eqref{eq:Matrix_triple3}, it follows that 
\begin{equation} \label{eq:Matrix_triple4}
-\sum_{j = 1}^d \int_{S^d} \frac{\del_j\rho}{\eta_j^2}\left(\del_jY_k^m\right)\cc{Y_k^n}\, d\sigma_d = \frac{1}{2}\int_{S^d} \left(\Delta_{S^d}\rho\right) Y_k^m\cc{Y_k^n}\, d\sigma_d. 
\end{equation} 
Finally, we substitute \eqref{eq:Matrix_triple4} and the expansion \eqref{eq:Domain} for $\rho$ into \eqref{eq:Matrix_triple1} to obtain 
\begin{align*} 
M_{m, n}^{(d, k)} & = -\frac{1}{2}\int_{S^d} \left(-\Delta_{S^d}\rho + 2k\rho\right)Y_k^m\cc{Y_k^n}\, d\sigma_d \\ 
& = -\frac{1}{2}\sum_{p = 0}^\infty \sum_{q = 0}^{N(d, p)} A_{p, q}\int_{S^d} \left(-\Delta_{S^d}Y_{p, q} + 2kY_{p, q}\right)Y_k^m\cc{Y_k^n}\, d\sigma_d \\ 
& \stackrel{\eqref{eq:SH_eigs}}{=} -\frac{1}{2}\sum_{p = 0}^\infty \sum_{q = 0}^{N(d, p)} \Big(p(p + d - 1) + 2k\Big)\int_{S^d} Y_{p, q}Y_k^m\cc{Y_k^n}\, d\sigma_d, 
\end{align*} 
which gives the desired result. 
\end{proof} 

In order to use Lemma \ref{thm:Matrix_triple}, we require the evaluation of $W^{p, k}_{q, m, n}$. We introduce additional notation that simplifies our presentation in deriving the explicit expression for $W^{p, k}_{q, m, n}$. For any $(d - 1)$-tuples $q, m, n$ satisfying \eqref{eq:SH_tuple} with $q_d = p$, $m_d = n_d = k$, we define the 3-tuple $T_j = (t_j^1, t_j^2, t_j^3)\coloneqq (q_j, m_j, n_j)$ and introduce the following variables for $j = 1, 2, \dots, d$:  
\begin{align*} 
s_j = q_j + m_j + n_j, \qquad \diff_j^i = t_j^i - t_{j - 1}^i, \qquad \nu_j = \frac{j - 1}{2}. 
\end{align*} 
Since the real and complex hyperspherical harmonics are both separable in hyperspherical coordinates (see Section \ref{sec:Harmonics}), it follows that $W^{p, k}_{q, m, n}$ can be written as a product of integrals  
\begin{equation*} \label{eq:Matrix_3j_prod}
W^{p, k}_{q, m, n} = I(T_1, T_2)\prod_{j = 3}^d I(T_{j - 1}, T_j), \ \ 1\le q\le N(d, p), \, 1\le m, n\le N(d, k), 
\end{equation*} 
where the integrals $I(T_1, T_2)$ and $I(T_{j - 1}, T_j)$, $j = 3, 4, \dots, d$, are given by 
\begin{align} 
\notag I(T_1, T_2) & = \int_0^{2\pi}\int_0^{\pi} \widetilde Y_{q_2, q_1}(\phi, \theta_2)\widetilde Y_{m_2}^{m_1}(\phi, \theta_2)\cc{\widetilde Y_{n_2}^{n_1}(\phi, \theta_2)}\sin(\theta_2)\, d\theta_2\, d\phi, \\ 
\notag I(T_{j - 1}, T_j) & = \int_0^{\pi} \left(\prod_{i = 1}^3 Y(\theta_j; t_{j - 1}^i, t_j^i)\right) \sin^{j - 1}(\theta_j)\, d\theta_j \\ 
\notag & = \left(\prod_{i = 1}^3 \mu_j^{(i)}\right)^{-1} \int_0^\pi \left(\prod_{i = 1}^3 C_{\diff_j^i}^{\left(t_{j - 1}^i + \nu_j\right)}(\cos\theta_j)\right) (\sin\theta_j)^{s_{j - 1} + 2\nu_j}\, d\theta_j \\ 
\label{eq:Matrix_3j_int} & = \left(\prod_{i = 1}^3 \mu_j^{(i)}\right)^{-1} \int_{-1}^1 (1 - z^2)^{\frac{s_{j - 1}}{2} + \nu_j - \frac{1}{2}} \prod_{i = 1}^3 C_{\diff_j^i}^{\left(t_{j - 1}^i + \nu_j\right)}(z)\, dz.  
\end{align} 
The constant $\mu_j^{(i)}$ for $i = 1, 2, 3$ and $j = 3, 4, \dots, d$ satisfies 
\begin{align*} 
\left(\mu_j^{(i)}\right)^2 & = \frac{4\pi\Gamma(t_j^i + t_{j - 1}^i + 2\nu_j)}{2^{2t_{j - 1}^i + j}\left(\diff_j^i\right)!\left(t_j^i + \nu_j\right)\Gamma^2(t_{j - 1}^i + \nu_j)}. 
\end{align*}

Our next theorem provides an explicit expression for these $(d - 1)$ integrals above involving the Pochhammer's symbol and the Wigner $3j$-symbol $\mat{j_1}{j_2}{j_3}{m_1}{m_2}{m_3}$; see \cite[Chapter 34]{NIST} for the definition of Wigner $3j$-symbols. A crucial property of the Wigner $3j$-symbol is the following: If $\mat{j_1}{j_2}{j_3}{m_1}{m_2}{m_3}\neq 0$, then all the following \emph{selection rules} must be satisfied: 
\begin{enumerate} 
\item $m_i\in\{-j_i, -j_i + 1, -j_i + 2, \dots, j_i\}$ for $i = 1, 2, 3$. 
\item $m_1 + m_2 + m_3 = 0$. 
\item The triangle conditions $\abs{j_1 - j_2}\le j_3\le j_1 + j_2$. 
\item $(j_1 + j_2 + j_3)\ge 0$ is an integer (and, moreover, an even integer if $m_1 = m_2 = m_3 = 0$). 
\end{enumerate}

\begin{theorem} \label{thm:Matrix_3j} 
Fix $d\ge 3$, $p\in\N$, and $k\in\Z^+$. Write $c_{T_2} = \sqrt{\frac{(2q_2 + 1)(2m_2 + 1)(2n_2 + 1)}{4\pi}}$ and $Q(q_1) = \mat{q_2}{m_2}{n_2}{q_1}{m_1}{-n_1}$. We have $I(T_1, T_2) = c_{T_2}\mat{q_2}{m_2}{n_2}{0}{0}{0}Q(T_1, T_2)$, where 
\begin{equation*} 
Q(T_1, T_2) = \begin{dcases} 
\, (-1)^{m_1}\delta_{m_1, n_1}\mat{q_2}{m_2}{n_2}{0}{m_1}{-m_1} & \ \ \tn{ if } q_1 = 0, \\ 
\, \frac{(-1)^{n_1}}{\sqrt{2}}\Big[Q(-q_1) + (-1)^{q_1}Q(q_1)\Big] & \ \ \tn{ if } q_1 > 0, \\ 
\, \frac{i(-1)^{n_1}}{\sqrt{2}}\Big[Q(q_1) - (-1)^{q_1}Q(-q_1)\Big] & \ \ \tn{ if } q_1 < 0.
\end{dcases} 
\end{equation*} 
For $j = 3, 4, \dots, d$, we have $I(T_{j - 1}, T_j) = \left(\mu_j^{(1)}\mu_j^{(2)}\mu_j^{(3)}\right)^{-1}H(T_{j - 1}, T_j)$, where 
\begin{align*} 
H(T_{j - 1}, T_j) & = \sum_{\substack{0\le \ell_1\le \lfloor \diff_j^1/2\rfloor \\ 0\le \ell_2\le \lfloor \diff_j^2/2\rfloor \\ 0\le \ell_3\le \lfloor \diff_j^3/2\rfloor}} \prod_{i = 1}^3 V(\diff_j^i, t_{j - 1}^i + \nu_j, \ell_i)\sum_{\substack{\tau_1, \tau_2\in\N \\\ \tau_2 \tn{ even}}} (2\tau_1 + 1)(2\tau_2 + 1) L(s_{j - 1}, \nu_j, \tau_2) \\ 
& \hspace{2.5cm} \times \mat{\diff_j^2 - 2\ell_2}{\diff_j^3 - 2\ell_3}{\tau_1}{0}{0}{0}^2\mat{\diff_j^1 - 2\ell_1}{\tau_1}{\tau_2}{0}{0}{0}^2.  
\end{align*} 
The constants $V$ and $L$ are defined by 
\begin{align*} 
V(\beta, \alpha, \ell) & = (1 + 2\beta - 4\ell)\frac{(\alpha)_{\beta - \ell}}{\left(\frac{3}{2}\right)_{\beta - \ell}}\frac{(\alpha - \frac{1}{2})_\ell}{\ell!}, \\ 
L(s_{j - 1}, \nu_j, \tau_2) & = \frac{\pi\Gamma^2\left(\frac{s_{j - 1}}{2} + \nu_j + \frac{1}{2}\right)}{\Gamma\left(\frac{s_{j - 1}}{2} + \nu_j + 1 + \frac{\tau_2}{2}\right)\Gamma\left(\frac{s_{j - 1}}{2} + \nu_j + \frac{1}{2} - \frac{\tau_2}{2}\right)\Gamma\left(\frac{\tau_2}{2} + 1\right)\Gamma\left(-\frac{\tau_2}{2} + \frac{1}{2}\right)}. 
\end{align*} 
Here, $\lfloor\cdot \rfloor$ and $(\alpha)_n = \frac{\Gamma(\alpha + n)}{\Gamma(\alpha)}$ denote the floor function and the Pochhammer's symbol, respectively. 
\end{theorem} 
\begin{proof} 
Write $d\sigma_2 = \sin\theta_2\, d\theta_2\, d\phi$ with $0\le \phi < 2\pi$ and $0\le \theta_2\le \pi$. Recall that the product of three complex three-dimensional spherical harmonics can be written in terms of the Wigner $3j$-symbol by \cite[Eq.~34.3.22]{NIST}: 
\begin{equation*} 
\int_{S^2(\phi, \theta_2)} \widetilde Y_{a_2}^{a_1}\widetilde Y_{b_2}^{b_1}\widetilde Y_{c_2}^{c_1}\, d\sigma_2 = \sqrt{\frac{(2a_2 + 1)(2b_2 + 1)(2c_2 + 1)}{4\pi}}\mat{a_2}{b_2}{c_2}{0}{0}{0}\mat{a_2}{b_2}{c_2}{a_1}{b_1}{c_1}. 
\end{equation*} 
Now, using the complex conjugate formula for the normalised three-dimensional complex spherical harmonics, we have that 
\begin{align*} 
I(T_1, T_2) = (-1)^{n_1}\int_{S^2(\phi, \theta_2)} \widetilde Y_{q_2, q_1}\widetilde Y_{m_2}^{m_1}\widetilde Y_{n_2}^{-n_1}\, d\sigma_2. 
\end{align*} 
If $q_1 = 0$, then 
\begin{align*} 
I(T_1, T_2) & = (-1)^{n_1}\int_{S^2} \widetilde Y_{q_2}^0\widetilde Y_{m_2}^{m_1}\widetilde Y_{n_2}^{-n_1}\, d\sigma_2 \\ 
& = (-1)^{n_1}c_{T_2}\mat{q_2}{m_2}{n_2}{0}{0}{0}\mat{q_2}{m_2}{n_2}{0}{m_1}{-n_1} \\ 
& = (-1)^{m_1}\delta_{m_1, n_1}c_{T_2}\mat{q_2}{m_2}{n_2}{0}{0}{0}\mat{q_2}{m_2}{n_2}{0}{m_1}{-m_1}. 
\end{align*} 
If $q_1 > 0$, then 
\begin{align*} 
I(T_1, T_2) & = \frac{(-1)^{n_1}}{\sqrt{2}}\int_{S^2} \Big[\widetilde Y_{q_2}^{-q_1}\widetilde Y_{m_2}^{m_1}\widetilde Y_{n_2}^{-n_1} + (-1)^{q_1}\widetilde Y_{q_2}^{q_1}\widetilde Y_{m_2}^{m_1}\widetilde Y_{n_2}^{-n_1}\Big]\, d\sigma_2 \\ 
& = \frac{(-1)^{n_1}}{\sqrt{2}}c_{T_2}\mat{q_2}{m_2}{n_2}{0}{0}{0}\left[\mat{q_2}{m_2}{n_2}{-q_1}{m_1}{-n_1} + (-1)^{q_1}\mat{q_2}{m_2}{n_2}{q_1}{m_1}{-n_1}\right]. 
\end{align*} 
If $q_1 < 0$, then 
\begin{align*} 
I(T_1, T_2) & = \frac{i(-1)^{n_1}}{\sqrt{2}}\int_{S^2} \Big[\widetilde Y_{q_2}^{q_1}\widetilde Y_{m_2}^{m_1}\widetilde Y_{n_2}^{-n_1} - (-1)^{q_1}\widetilde Y_{q_2}^{-q_1}\widetilde Y_{m_2}^{m_1}\widetilde Y_{n_2}^{-n_1}\Big]\, d\sigma_2 \\ 
& = \frac{i(-1)^{n_1}}{\sqrt{2}}c_{T_2}\mat{q_2}{m_2}{n_2}{0}{0}{0}\left[\mat{q_2}{m_2}{n_2}{q_1}{m_1}{-n_1} - (-1)^{q_1}\mat{q_2}{m_2}{n_2}{-q_1}{m_1}{-n_1}\right]. 
\end{align*} 
This completes the proof for $I(T_1, T_2)$. 

To establish the formula for $I(T_{j - 1}, T_j)$, we need only show the integral in \eqref{eq:Matrix_3j_int} is equal to $H(T_{j - 1}, T_j)$. We follow the proof in \cite[Section V]{wen1985}. Let $P_\beta(z)$ and $(\alpha)_\beta = \Gamma(\alpha + \beta)/\Gamma(\alpha)$ denote the Legendre polynomial of degree $\beta$ and the Pochhammer's symbol, respectively. The first step is to combine the connection sum formula \cite[Eq.~18.18.16]{NIST} for Gegenbauer polnomials with $\lambda = 1/2$ together with the fact that $C_\beta^{(1/2)}(z) = P_\beta(z)$ \cite[Eq.~18.7.8]{NIST}: 
\begin{equation*} 
C_\beta^{(\alpha)}(z) = \sum_{\ell = 0}^{\lfloor \beta/2\rfloor} V(\beta, \alpha, \ell)P_{\beta - 2\ell}(z), \ \ \tn{ where } V(\beta, \alpha, \ell)\coloneqq (1 + 2\beta - 4\ell)\frac{(\alpha)_{\beta - \ell}}{(\frac{3}{2})_{\beta - \ell}}\frac{(\alpha - \frac{1}{2})_\ell}{\ell!}. 
\end{equation*}
Fix $j\in\{3, 4, \dots, d\}$ and choose $\beta_j^i = \diff_j^i$ and $\alpha_j^i = t_{j - 1}^i + \nu_j$. The integral in \eqref{eq:Matrix_3j_int} is equal to 
\begin{gather} \label{eq:Matrix_3j1}
\begin{aligned} 
\int_{-1}^1 & (1 - z^2)^{\frac{s_{j - 1}}{2} + \nu_j - \frac{1}{2}}\left(\prod_{i = 1}^3 \sum_{\ell_i = 0}^{\lfloor \beta_j^i/2\rfloor} V(\beta_j^i, \alpha_j^i, \ell_i)P_{\beta_j^i - 2\ell_i}(z)\right) dz \\ 
& = \sum_{\substack{0\le \ell_1\le \lfloor \beta_j^1/2\rfloor \\ 0\le \ell_2\le \lfloor \beta_j^2/2\rfloor \\ 0\le \ell_3\le \lfloor \beta_j^3/2\rfloor}} \prod_{i = 1}^3 V(\beta_j^i, \alpha_j^i, \ell_i)\int_{-1}^1 (1 - z^2)^{\frac{s_{j - 1}}{2} + \nu_j - \frac{1}{2}} \prod_{i = 1} P_{\beta_j^i - 2\ell_i}(z)\, dz. 
\end{aligned} 
\end{gather} 

To evaluate the integral involving the triple product of Legendre polynomials, we use the fact that the product of Legendre polynomials can be written in terms of the Wigner $3j$-symbol by \cite[Eq.~34.3.19]{NIST} 
\begin{align*} 
P_{b_1}(z)P_{b_2}(z) = \sum_{\tau_1\in\N} (2\tau_1 + 1)\mat{b_1}{b_2}{\tau_1}{0}{0}{0}^2P_{\tau_1}(z). 
\end{align*} 
Applying this identity twice and using the fact that odd permutations of columns of Wigner $3j$-symbols produce a phase factor \cite[Eq.~34.3.9]{NIST}, we obtain 
\begin{gather} \label{eq:Matrix_3j2} 
\begin{aligned} 
& \prod_{i = 1}^3 P_{\beta_j^i - 2\ell_i}(z) = P_{\beta_j^1 - 2\ell_1}(z) \sum_{\tau_1\in\N} (2\tau_1 + 1)\mat{\beta_j^2 - 2\ell_2}{\beta_j^3 - 2\ell_3}{\tau_1}{0}{0}{0}^2 P_{\tau_1}(z) \\ 
& = \sum_{\tau_1, \tau_2\in\N} (2\tau_1 + 1)(2\tau_2 + 1)\mat{\beta_j^2 - 2\ell_2}{\beta_j^3 - 2\ell_3}{\tau_1}{0}{0}{0}^2\mat{\beta_j^1 - 2\ell_1}{\tau_1}{\tau_2}{0}{0}{0}^2 P_{\tau_2}(z). 
\end{aligned} 
\end{gather} 
Substituting \eqref{eq:Matrix_3j2} into \eqref{eq:Matrix_3j1} and comparing the resulting expression with the given expression for $H(T_{j - 1}, T_j)$, we need only show  
\begin{align*} 
\int_{-1}^1 (1 - z^2)^{\frac{s_{j - 1}}{2} + \nu_j - \frac{1}{2}}P_{\tau_2}(z)\, dz = L(s_{j - 1}, \nu_j, \tau_2). 
\end{align*} 
This follows from applying the following integration formula for Legendre polynomials with $\gamma = (s_{j - 1} + 2\nu_j + 1)/2$ \cite[Eq.~7.132.1 with $\mu = 0$]{Gradshteyn}:  
\begin{align*} 
\int_{-1}^1 (1 - z^2)^{\gamma - 1}P_{\tau_2}(z)\, dz = \frac{\pi\Gamma^2(\gamma)}{\Gamma\left(\gamma + \frac{\tau_2}{2} + \frac{1}{2}\right)\Gamma\left(\gamma - \frac{\tau_2}{2}\right)\Gamma\left(\frac{\tau_2}{2} + 1\right)\Gamma\left(-\frac{\tau_2}{2} + \frac{1}{2}\right)}, \ \ \mathrm{Re}(\gamma) > 0. 
\end{align*} 
Finally, we may assume that $\tau_2$ is even, since otherwise $(1 - z^2)^{(s_{j - 1} + 2\nu_j - 1)/2}P_{\tau_2}(z)$ is an odd function which results in $L(s_{j - 1}, \nu_j, \tau_2) = 0$. 
\end{proof} 

We now combine Lemma \ref{thm:Matrix_triple} and Theorem \ref{thm:Matrix_3j} to prove Theorem \ref{thm:Iso_notball}. 

\begin{proof}[Proof of Theorem \ref{thm:Iso_notball}] 
Choose the perturbation function $\rho = Y_{2, q}$ with the tuple $q = (0, 2, \dots, 2)$. With this choice of $\rho$, Lemma \ref{thm:Matrix_trace} tells us that the trace of $M^{(d, k)}$ is 0. Since $M^{(d, k)}$ is Hermitian which is diagonalisable, it suffices to show that $M^{(d, k)}$ has a nonzero entry. In that case, we must have $\lambda_{k, N(d, k)}^{(1)} = \lambda_{N_{d, k + 1}}^{(1)} > 0$ since we defined it to be the largest eigenvalue of $M^{(d, k)}$. 

We claim that the diagonal entry of $M^{(d, k)}$ corresponding to the tuple $m = (k, k, \dots, k)$ is one such nonzero entry. From Lemma \ref{thm:Matrix_triple}, we need only show that $W^{2, k}_{q, m, m}\neq 0$ with our choice of tuples. Recall the definition of the 3-tuple $T_j = (q_j, m_j, n_j)$ for $j = 1, 2, \dots, d$. From Theorem \ref{thm:Matrix_3j}, we have $c_{T_2} = \sqrt{\frac{5}{4\pi}}(2k + 1) \neq 0$ and it follows from \cite[Eqs.~34.3.5 \& 34.3.7]{NIST} that 
\begin{align*} 
I(T_1, T_2) & = (-1)^k c_{T_2}\mat{2}{k}{k}{0}{0}{0}\mat{2}{k}{k}{0}{k}{-k} \neq 0. 
\end{align*} 
Next, observe that for each $j = 3, 4, \dots, d$, we have 
\begin{equation*} 
\diff_j^i = t_j^i - t_{j - 1}^i = 0, \ \ i = 1, 2, 3, \ \ \tn{ and } \ \ s_{j - 1} = q_{j - 1} + m_{j - 1} + n_{j - 1} = 2 + 2k. 
\end{equation*} 
For a fixed $j\in\{3, 4, \dots, d\}$, Theorem \ref{thm:Matrix_3j} gives 
\begin{align*} 
I(T_{j - 1}, T_j) & = \frac{V(0, 2 + \nu_j, 0)V(0, k + \nu_j, 0)^2}{\mu_j^{(1)}\mu_j^{(2)}\mu_j^{(3)}}\sum_{\substack{\tau_1, \tau_2\in\N \\\ \tau_2 \tn{ even}}} (2\tau_1 + 1)(2\tau_2 + 1)L(2 + 2k, \nu_j, \tau_2) \\ 
& \hspace{5cm} \times \mat{0}{0}{\tau_1}{0}{0}{0}^2\mat{0}{\tau_1}{\tau_2}{0}{0}{0}^2. 
\end{align*} 
Since the Pochhammer's symbol satisfies $(\alpha)_0 =1$ for any $\alpha\ge 0$, we see that $V(0, 2 + \nu_j, 0)$ and $V(0, k + \nu_j, 0)$ are both 1. Next, using the triangle conditions for Wigner $3j$-symbols, we must have $\tau_1 = 0$ and subsequently $\tau_2 = 0$. Recalling $\nu_j = (j - 1)/2$ and using the fact that $\mat{0}{0}{0}{0}{0}{0} = 1$ (see \cite[Eq.~34.3.1]{NIST}), we find 
\begin{align*} 
I(T_{j - 1}, T_j) & = \frac{L(2 + 2k, \nu_j, 0)}{\mu_j^{(1)}\mu_j^{(2)}\mu_j^{(3)}}\mat{0}{0}{0}{0}{0}{0}^4 = \frac{\sqrt{\pi}\Gamma\left(k + 1 + \frac{j}{2}\right)}{\mu_j^{(1)}\mu_j^{(2)}\mu_j^{(3)}\Gamma\left(k + \frac{j + 3}{2}\right)} > 0.  
\end{align*} 
This completes the proof since $W^{2, k}_{q, m, m}$ is a product of $I(T_1, T_2)$ and $\{I(T_{j - 1}, T_j)\}_{j = 3}^d$. 
\end{proof}

Following \cite[Corollary 2.3]{viator2022}, we wish to establish the first-order behaviour of the Steklov eigenvalues for a nearly hyperspherical domain $\Omega_\vareps$ of the form \eqref{eq:Domain} with $\rho = Y_{p, q}(\hat\theta)$, for both $p\in\N$ and the $(d - 1)$-tuple $q$ fixed. The proof in \cite{viator2022} is based on a delicate analysis of the integral $W^{p, k}_{q, m, n}$ using the selection rules for Wigner $3j$-symbols, where the authors proved that the entry \eqref{eq:Matrix_triple} for $M^{(d, k)}$ involving the infinite sum reduced to a finite sum over the set $\{p\le 2k, p\tn{ even}\}$.

\begin{lemma} \label{thm:Matrix_triple_finite} 
Given $d\ge 3$ and $k\in\Z^+$, let $M^{(d, k)}$ and $\rho$ be the Hermitian matrix and the perturbation function defined in Theorem \ref{thm:Oeps_matrix} and \eqref{eq:Domain}, respectively. Then the entries of $M^{(d, k)}$ can be written as 
\begin{equation} 
M_{m ,n}^{(d, k)} = -\frac{1}{2}\sum_{\substack{p = 0 \\ p \tn{ even}}}^\infty\sum_{q = 1}^{N(d, k)} A_{p, q}\Big(p(p + d - 1) + 2k\Big)W_{q, m, n}^{p, k}. 
\end{equation} 
\end{lemma} 
\begin{proof} 
Fix $p\in\N$ and $k\in\Z^+$. We claim that if $p$ is odd, then $W^{p, k}_{q, m, n} = 0$ for all $q, m, n$ satisfying \eqref{eq:SH_tuple}. Recall that $s_j = q_j + m_j + n_j$ for $j = 1, 2, \dots, d$. Looking at the formula for $I(T_1, T_2)$ in Theorem \ref{thm:Matrix_3j}, we see that $W^{p, k}_{q, m, n} = 0$ if $\mat{q_2}{m_2}{n_2}{0}{0}{0}$ is zero. Thus, we may assume $\mat{q_2}{m_2}{n_2}{0}{0}{0} \neq 0$ without loss of generality. By the fourth selection rule for Wigner $3j$-symbols, we must have $s_2 = q_2 + m_2 + n_2$ even. 

Next, we turn our attention to $I(T_2, T_3)$. Recall that $\diff_j^i = t_j^i - t_{j - 1}^i$ and $\tau_2$ is even. The Wigner $3j$ symbols appearing in $I(T_2, T_3)$ have the form 
\begin{equation*} 
\mat{\diff_3^2 - 2\ell_2}{\diff_3^3 - 2\ell_3}{\tau_1}{0}{0}{0}^2\mat{\diff_1^1 - 2\ell_1}{\tau_1}{\tau_2}{0}{0}{0}^2. 
\end{equation*} 

Looking at the Wigner $3j$-symbols from $I(T_2, T_3)$, the fourth selection rule tells us that we may assume both $\diff_3^1 + \tau_1$ and $(\diff_3^2 - 2\ell_2 + \diff_3^3 - 2\ell_3 + \tau_1)$ are even. In this case, 
\begin{align*} 
\diff_3^1 + \tau_1 & = q_3 - q_2 + \tau_1 \\ 
\diff_3^2 - 2\ell_2 + \diff_3^3 - 2\ell_3 + \tau_1 & = m_3 - m_2 - 2\ell_2 + n_3 - n_2 - 2\ell_3 + \tau_1 \\ 
& = (s_3 - q_3) - (s_2 - q_2) + \tau_1 - 2\ell_2 - 2\ell_3 \\ 
& = (s_3 - s_2) - (\diff_3^1 - \tau_1) - 2\ell_2 - 2\ell_3.
\end{align*} 
Since $\diff_3^1 - \tau_1 = \diff_3^1 + \tau_1 - 2\tau_1$ and $s_2$ are both even, we may conclude that $s_3$ is even as well. 

Continuing inductively, we conclude that $\diff_j^1 + \tau_1$ and $s_j$ is even (so that $I(T_{j - 1}, T_j)$ is nonzero) for each $j = 3, 4, \dots, d - 1$.  We now consider $I(T_{d-1}, T_d)$.  By the arguments above, we can assume that $\diff_d^1 + \tau_1$ is even, which in turn implies that $\diff_d^1 - \tau_1$ is also even. Thus, we compute 
\begin{align*} 
\diff_d^2 - 2\ell_{d-1} + \diff_d^3 - 2\ell_d + \tau_1 & = m_d - m_{d-1} - 2\ell_{d-1} + n_d - n_{d-1} - 2\ell_d + \tau_1 \\
& = (s_d - q_d) - (s_{d-1} - q_{d-1}) + \tau_1 - 2\ell_{d-1} - 2\ell_d \\
& = (s_d - s_{d-1}) - (\diff_d^1 - \tau_1) - 2\ell_{d-1} - 2\ell_d \\
& = p + 2k - s_{d-1} - (\diff_d^1 - \tau_1) - 2\ell_{d-1} - 2\ell_d ,
\end{align*}
where we have used that that $(t_d^1 , t_d^2 , t_d^3) = (q_d, m_d, n_d) = (p, k ,k)$.  But if $p$ is odd, then $\diff_d^2 - 2\ell_{d-1} + \diff_d^3 - 2\ell_d + \tau_1$ must be odd as well.  By the fourth selection rule, we conclude that $I(T_{d-1},T_d)=0$, so that $W^{p,k}_{q,m,n} = 0$ as well, completing the proof.
\end{proof} 

The following theorem is immediate from Corollary \ref{thm:Matrix_sum} and Lemma \ref{thm:Matrix_triple_finite}. 

\begin{theorem} \label{thm:Ypq} 
Given $d\ge 3$, consider a nearly hyperspherical domain $\Omega_\vareps$ of the form \eqref{eq:Domain} with $A_{p, q} = \delta_{p, p'}\delta_{q, q'}$. For any positive integer $k\in\Z^+$ and any $j = 1, 2, \dots, N(d, k)$, 
\begin{enumerate} 
\item if $p' = 0$ and $q' = (0, 0, \dots, 0)$ is the trivial $(d - 1)$-tuple, then $\lambda_{k, j}^{(1)} = -k\abs{S^d}^{-1/2}$. 
\item if $p'$ is odd, then the Steklov eigenvalue $\lambda_{k, j}^\vareps$ is unperturbed at first-order in $\vareps$. 
\end{enumerate} 
\end{theorem} 

Reducing the infinite sum \eqref{eq:Matrix_triple} to the case where $p\le 2k$ is intractable for $d\ge 3$. More specifically, we compute the entry of the matrix $M^{(d, k)}$ using \eqref{eq:Matrix_triple} for $p = 2k + 2\xi$ with $\xi\in\Z^+$, $q = (0, 2k, \dots, 2k)$, $m = n = (k, k, \dots, k)$, and we see numerically that this particular entry is zero due to cancellations of Wigner $3j$-symbols appearing in the formula of $W^{p, k}_{q, m, n}$ (see Theorem \ref{thm:Matrix_3j}).

\subsubsection*{Acknowledgements.} We would like to express our gratitude to Chiu-Yen Kao and Nathan Schroeder for sharing their helpful insights into the Steklov eigenvalue problem in arbitrary dimensions. We would also like to thank Yerim Kone, Lucas Alland, and Amy Liu for their work in 2D Steklov shape perturbation, and Swarthmore College for funding their summer work.

\printbibliography 

\end{document}